\documentclass[reqno,11pt]{amsart}
\usepackage[letterpaper,margin=1in,footskip=0.25in]{geometry}
\usepackage[utf8]{inputenc}
\usepackage{mathptmx}
\usepackage[cal=boondoxupr]{mathalfa}
\usepackage{mathrsfs}
\usepackage{setspace, amssymb, amsopn, amsmath, array, pdfsync, url, amsfonts, float,enumitem,microtype}
\usepackage[dvipsnames]{xcolor}
\definecolor{chianti}{rgb}{0.6,0,0}
\definecolor{meretale}{rgb}{0,0,.6}
\definecolor{leaf}{rgb}{0,.35,0}

\usepackage{comment}

\usepackage[colorlinks=true, hyperindex, citecolor=meretale, urlcolor=leaf, linkcolor=chianti]{hyperref}

\usepackage[all]{xy}
\usepackage{tikz, tikz-cd}
\usetikzlibrary{decorations.markings}
\usepackage{fancyvrb,newverbs}
\usepackage{arydshln}
\usepackage[capitalise]{cleveref}
\usepackage{multicol}
\usepackage{booktabs}

\allowdisplaybreaks

\numberwithin{equation}{section}
\newtheorem{Theoremx}{Theorem}
 
\newtheorem{theorem}{Theorem}[section]
\newtheorem{lemma}[theorem]{Lemma}
\newtheorem{proposition}[theorem]{Proposition}
\theoremstyle{definition}
\newtheorem{setting}[theorem]{Setting}
\theoremstyle{definition}

\newtheorem{corollary}[theorem]{Corollary}
\newtheorem{claim}[theorem]{Claim}

\theoremstyle{definition}
\newtheorem{definition}[theorem]{Definition}
\newtheorem{remark}[theorem]{Remark}
\theoremstyle{remark}

\newtheorem{example}[theorem]{Example}

\renewcommand{\ker}{\operatorname{ker}}

\newcommand{\Jac}{\operatorname{Jac}}

\newcommand{\Spec}{\operatorname{Spec}}
\newcommand{\ts}{\textsuperscript}
\newcommand{\maxSpec}{\operatorname{maxSpec}}

\usepackage{stmaryrd}

\newcommand{\Ht}{\operatorname{ht}}
\newcommand{\Height}{\operatorname{ht}}

\newcommand{\id}{\operatorname{id}}
\newcommand{\Ext}{\operatorname{Ext}}

\newcommand{\Supp}{\operatorname{Supp}}

\newcommand{\Hom}{\operatorname{Hom}}

\newcommand{\NN}{\mathbb{N}}

\newcommand{\coker}{\operatorname{coker}}

\newcommand{\Tr}{\operatorname{Tr}}

\newcommand{\Div}{\operatorname{div}}

\newcommand{\Frac}{\operatorname{Frac}}

\newcommand{\N}{\mathbb{N}}

\newcommand{\Z}{\mathbb{Z}}
\newcommand{\F}{\mathbb{F}}
\newcommand{\Q}{\mathbb{Q}}
\newcommand{\PP}{\mathbb{P}}

\newcommand{\fm}{\mathfrak{m}}
\newcommand{\fp}{\mathfrak{p}}
\newcommand{\fq}{\mathfrak{q}}

\newcommand{\fn}{\mathfrak{n}}

\newcommand{\fP}{\mathfrak{P}}

\newcommand{\fQ}{\mathfrak{Q}}

\newcommand{\cO}{\mathcal{O}}

\newcommand{\cP}{\mathcal{P}}

\newcommand{\kperf}{K^{\infty}}

\makeatother

\crefname{Theoremx}{Theorem}{Theorems}
\crefname{setting}{Setting}{Settings}
\crefname{lemma}{Lemma}{Lemmas}

\makeatletter
\renewcommand*{\eqref}[1]{%
  \hyperref[{#1}]{\textup{\tagform@{\ref*{#1}}}}%
}
\makeatother

\usepackage[backend=biber, style=alphabetic, datamodel=mrnumber,sorting=anyt,maxalphanames=5,maxbibnames=99]{biblatex}
\usepackage{xurl}
\hypersetup{breaklinks=true}

\DeclareFieldFormat{mrnumber}{
  MR\addcolon\space
  \ifhyperref
    {\href{http://www.ams.org/mathscinet-getitem?mr=#1}{\nolinkurl{#1}}}
    {\nolinkurl{#1}}}

\renewbibmacro*{doi+eprint+url}{
  \iftoggle{bbx:doi}
    {\printfield{doi}}
    {}
  \newunit\newblock
  \printfield{mrnumber}
  \newunit\newblock
  \iftoggle{bbx:eprint}
    {\usebibmacro{eprint}}
    {}
  \newunit\newblock
  \iftoggle{bbx:url}
    {\usebibmacro{url+urldate}}
    {}}

\renewbibmacro{in:}{}

\begin{document}
\title{Bertini's theorem for \texorpdfstring{$F$}{F}-rational \texorpdfstring{$F$}{F}-pure singularities}

\author[De Stefani]{Alessandro De Stefani}
\address{Dipartimento di Matematica, Universit\`{a} di Genova, Via Dodecaneso 35, 16146 Genova,
Italy}
\email{alessandro.destefani@unige.it}
\urladdr{\url{https://a-destefani.github.io/ads/}}

\author[Polstra]{Thomas Polstra}
\address{Department of Mathematics, University of Alabama, Tuscaloosa, AL 35487 USA}
\email{tmpolstra@ua.edu}
\urladdr{\url{https://thomaspolstra.github.io/}}

\author[Simpson]{Austyn Simpson}
\address{Department of Mathematics, Bates College, 3 Andrews Rd, Lewiston, ME 04240 USA}
\email{asimpson2@bates.edu}
\urladdr{\url{https://austynsimpson.github.io/}}

\begin{abstract}
Let $k$ be an algebraically closed field of characteristic $p>0$, and let $X\subseteq\PP^n_k$ be a quasi-projective variety that is $F$-rational and $F$-pure. We prove that if $H \subseteq \PP^n_k$ is a general hyperplane, then $X \cap H$ is also $F$-rational and $F$-pure. Of related but independent interest, we present a relationship between the characteristic and index of a $\Q$-Gorenstein variety with isolated non-$F$-regular locus which is $F$-pure but not $F$-regular.
\end{abstract}

\maketitle

\section{Introduction}

A fundamental procedure for analyzing the properties of a quasi-projective variety $X\subseteq \PP^n_k$ is through induction on dimension, considering instead the section $X\cap H$ for a general hyperplane $H\subseteq \PP^n_k$. It is often crucial for applications, however, for the salient properties of $X$ to persist in the hyperplane section $X\cap H$. In case $X$ is smooth, this is accomplished by the classical Bertini theorem, \cite[II, Theorem 8.18]{Har77}. Outside of the smooth case, there is a large suite of results---collectively referred to as \emph{Bertini theorems}---asserting that the singularities of $X\cap H$ are no worse than those of $X$.

In equal characteristic zero, Bertini theorems for the singularities of the Minimal Model Program (for example, rational, log canonical, and log terminal singularities) are well-known; see e.g. \cite[\S 4]{Kol97}, \cite[Chapter 5]{KM98}, and \cite[\S 9]{Laz04}. For their counterparts in positive characteristic (that is, for $F$-singularities), this line of inquiry was initiated by Schwede and Zhang \cite{SZ13} showing that the notions of strong $F$-regularity and $F$-purity are inherited from a variety by a general hyperplane. They also show that this phenomenon can fail if $X$ is merely assumed to be $F$-injective. On the other hand, the question of whether $F$-rationality satisfies similar types of Bertini theorems remains open. Our main theorem is an answer to this question under the hypothesis that $X$ is also $F$-pure.

\begin{Theoremx}[Bertini's theorem for $F$-purity + $F$-rationality]
    \label{thm: Bertini for F-rational singularities}
Let $X$ be a variety over an algebraically closed field $k=\overline{k}$ of characteristic $p>0$.
    \begin{enumerate}[label=(\roman*)]
        \item Suppose $X$ is projective. If $X$ is $F$-rational and $F$-pure, then so is a general element of the linear system $|\mathcal{L}|$ where $\mathcal{L}$ is a very ample line bundle.
        \item More generally, if $\phi\colon X\to \mathbb{P}^n_k$ is a $k$-morphism with separably generated (not necessarily algebraic) residue field extensions and $X$ is $F$-rational and $F$-pure, then there exists a non-empty open subset $U\subseteq (\mathbb{P}_k^n)^{\vee}$ so that for each hyperplane $H\in U$, $\phi^{-1}(H)$ is $F$-rational and $F$-pure.
    \end{enumerate} 
\end{Theoremx}

Similar to \cite{CRST21,DS22,SZ13}, our proof of \cref{thm: Bertini for F-rational singularities} is inspired by the methods introduced in \cite{CGM86} (see also \cite{Spr98}). Consider the following three axioms for a property $\cP$ of locally noetherian schemes. 

\begin{enumerate}[label=(A\arabic*)]
    \item Let $\varphi\colon Y\to Z$ be a flat morphism of $F$-finite schemes with regular fibers. If $Z$ is $\cP$ then $Y$ is $\cP$.\label{A1:intro}
    \item Let $\psi\colon Y\to S$ be a finite type morphism of $F$-finite schemes where $S$ is integral with generic point $\eta\in S$. If $Y_\eta$ is geometrically $\cP$, then there exists an open neighborhood $\eta\in U\subseteq S$ such that $Y_s$ is geometrically $\cP$ for every $s\in U$.\label{A2:intro}
    \item $\cP$ is open on schemes of finite type over a field.\label{A3:intro}
\end{enumerate}
It is then shown in \cite[Theorem 1]{CGM86} that $\cP$ transfers from a variety to its general hyperplane sections (interpreted, for example, as in \cref{thm: Bertini for F-rational singularities}) provided that $\cP$ enjoys axioms \ref{A1:intro} and \ref{A2:intro}. If $\cP$ further satisfies \ref{A3:intro} then there is a host of more general results concerning other types of linear systems, for example those which do not arise from closed immersions. We postpone these more general statements, including \emph{Bertini's second theorem}, to \cref{sec:Bertini}.

Let us briefly summarize our proof strategy which involves applying a modification of the above axioms. Let $X\subseteq\PP^n_k$ be a $k$-variety and let $Z$ be the closure of the incidence correspondence $$\{(x,H)\in \PP^n_k\times_k (\PP^n_k)^{\vee}\mid x\in H\}$$ where $(\PP^n_k)^{\vee}$ is the dual space of hyperplanes with generic point $\eta$. In essence, the axiom \ref{A1:intro} is employed in \cite{CGM86} to show that $(X\times_{\PP^n_k} Z_\eta)\otimes_{\kappa(\eta)} L$ is $\cP$ for all finite extensions $\kappa(\eta)\subseteq L$ whenever $X$ satisfies $\cP$.

Often when proving that a local property $\cP$ satisfies \ref{A1:intro}, one reduces to the local setting and proves:
\begin{enumerate}[label=($\star$)]
    \item Let $(R,\fm)\to (S,\fn)$ be a flat local ring homomorphism between excellent local rings such that the closed fiber $S/\fm S$ is a regular local ring. If $R$ is $\cP$ then $S$ is $\cP$.\label{axiom:star}
\end{enumerate}

In practice, this condition often fails in prime characteristic when $\cP$ is some mild class of singularity due to the existence of purely inseparable residue field extensions. For example, as indicated in the table below, both $F$-injectivity and $F$-rationality can fail to satisfy \ref{axiom:star}. Bertini's theorem is known to be false for $F$-injectivity \cite[Theorem 7.5]{SZ13} and although the most general version of the $F$-rational analog is still open, the examples of \cite{QGSS24} show that the strategy of either \cite{CGM86} or of the present paper cannot be applied directly to give a positive answer.

\begin{table}[htp!]
\centering
\begin{tabular}{c llll}
    \toprule
    $\cP$&\multicolumn{2}{c}{\ref{axiom:star}}&\multicolumn{2}{c}{Bertini's Second Theorem}\\
    \cmidrule(lr){1-1} \cmidrule(lr){2-3} \cmidrule(lr){4-5}
    strongly $F$-regular&True&\cite[Corollary 4.6]{SZ13}&True&\cite[Theorem 6.1]{SZ13}\\
    $F$-pure&True&\cite[Proposition 4.8]{SZ13}&True&\cite[Theorem 6.1]{SZ13}\\
    $F$-injective&False&\cite[Proposition 4.2]{Ene09}&False&\cite[Theorem 7.5]{SZ13}\\
    $F$-rational&False&\cite[Theorem 1.1]{QGSS24}&\multicolumn{2}{c}{\textbf{open}}\\
    \bottomrule
    \end{tabular}
\end{table}

On the other hand, the rings constructed in \cite{QGSS24} inform our strategy for proving \cref{thm: Bertini for F-rational singularities} in that they are \emph{not} $F$-pure. More concretely, standard graded $\F_p(t)$-algebras $R$ are constructed in \emph{op. cit.} in which $R$ is $F$-rational but $R\otimes_{\F_p(t)} \F_p(t^{1/p})$ is not even $F$-injective. In the case of interest in this paper where $\cP= $ ``$F$-rational and $F$-pure," we are able to replace \ref{A1:intro} with the following condition which is still sufficient for proving Bertini's theorem (see \cref{thm:CGM}):

\begin{enumerate}[label=(B\arabic*)]
    \item Let $Y$ be a scheme of finite type over a field $K$, and let $K \subseteq L$ be a finite field extension. Suppose $X$ is a noetherian excellent scheme such that there exists a flat morphism $Y \times_K L \to X$ with regular fibers. If $X$ and $Y$ are $\cP$ then so is $Y \times_K L$. \label{CGM-B1-intro}
\end{enumerate}

With the notation as above, if we further require that $R$ be $F$-pure then we show in this paper that the base change $R\otimes_{\F_p(t)} \F_p(t^{1/p})$ \textit{will} be $F$-rational and $F$-pure, and in fact that the axiom \ref{CGM-B1-intro} holds for this property (\cref{B1:F-rational-F-pure-v2}). We then show that a version of the ``spreading out" axiom \ref{A2:intro} holds for $F$-rationality to conclude that a general hyperplane section of $X$ is $F$-rational and $F$-pure.

Since strong $F$-regularity is preserved by a general hyperplane section by \cite[Theorem 6.1]{SZ13} and implies both $F$-rationality and $F$-purity, the only novel situations to which \cref{thm: Bertini for F-rational singularities} may be applied concern non-$F$-regular varieties which are $F$-pure and $F$-rational. We emphasize that such examples are abundant, even in the case of surfaces (see \cite{Wat91}). For concreteness, consider:

\begin{example} Let $k$ be an algebraically closed field of characteristic $p\equiv 1\pmod 3$ and let $\gamma\in k$ be a primitive cube root of unity. Then the invariant subring
\begin{align*}
    R = \left(\frac{k[x,y,z]}{(x^3-yz(y+z))}\right)^{\Z/3\Z}
\end{align*}
is $F$-rational, $F$-pure, but not $F$-regular \cite[Example 7.16]{HH94b}. Here, the action of $\Z/3\Z$ is taken to be the one generated by $$\sigma\colon x\mapsto x,y\mapsto \gamma y, z\mapsto \gamma z.$$
\end{example} The above example is $\Q$-Gorenstein of index three, and the assumption that $p\equiv 1\pmod 3$ is necessary for the ring to be $F$-pure. We find this to be an instance of a broader phenomenon witnessed by $\Q$-Gorenstein isolated singularities satisfying the assumptions of \cref{thm: Bertini for F-rational singularities}.

\begin{Theoremx}\label{mainthm:index} (= \cref{prop: char cong to 1 mod p,cor:f-rational-not-f-regular})
    Let $(R,\fm,k)$ be an $F$-finite normal domain of prime characteristic $p>0$ and $\mathbb{Q}$-Gorenstein of index $n$. If $R$ is $F$-pure, not $F$-regular, and if $R_{\fp}$ is $F$-regular for all non-maximal prime ideals $\fp\in\Spec(R)$ (for example, if $R$ has an isolated singularity), then $p\equiv 1\pmod{n}$. If $R$ is further assumed to be $F$-rational, then $p\neq 2$.
\end{Theoremx}

The varieties we are considering in \cref{thm: Bertini for F-rational singularities} should be thought of as the counterparts of varieties over the complex numbers which are log canonical and rational but \emph{not} log terminal.\footnote{In the literature this notion is sometimes referred to as \emph{strict log canonicity}; see \cite{Fuj16} and the references therein.} In particular, \cref{mainthm:index} should be compared with the work of Ishii \cite{Ish00} which finds certain restrictions on the $\Q$-Gorenstein index of an isolated strictly log canonical singularity.

\subsection*{Acknowledgments} We thank Karl Schwede for inspiring conversations and for explaining many features of \cite{PSZ18} to us. We are also grateful to Linquan Ma for pointing out an inaccuracy in a previous version of this article and for helpful discussions. We thank the referee for a careful reading of the paper and for several suggestions that substantially improved its readability.

De Stefani was partially supported by the MIUR Excellence Department Project CUP D33C23001110001, PRIN 2022 Project 2022K48YYP, and by INdAM-GNSAGA. Polstra was supported by a grant from the National Science Foundation DMS \#2502317. Simpson was supported by NSF postdoctoral fellowship DMS \#2202890.

\section{Preliminaries}
\subsection{Notation and definitions}
Throughout this paper, all schemes $X$ are assumed to be noetherian, separated, and have characteristic $p$ where $p>0$ is a prime integer. For a point $x\in X$, we denote by $\kappa(x)$ the residue field of the local ring $\cO_{X,x}$. We say that $X$ is \emph{$F$-finite} if the $e$\ts{th} iterated (global) Frobenius $F^e\colon X\to X$ is a finite morphism for some or all $e>0$. By a \emph{variety}, we mean an integral scheme $X$ which is separated and of finite type over an algebraically closed field. In particular, all varieties in this article are $F$-finite.

We say that a property $\cP$ of noetherian local rings is a \emph{local property of schemes} if $X$ satisfies $\cP$ if and only if $\cO_{X,x}$ satisfies $\cP$ for all $x\in X$. Moreover, if $X\to\Spec k$ is a map of schemes where $k$ is a field, then we say that $X$ is \emph{geometrically $\cP$ over $k$} if $X\times_k \ell$ is $\cP$ for every finite extension $k\subseteq \ell$. If there is no ambiguity as to which field $k$ we are referencing, we will simply say that $X$ is \emph{geometrically $\cP$} in such a circumstance. Note that $F$-finiteness of $X$ is preserved by finite base changes $X\times_k \ell$.

We now recall the various notions of $F$-singularities of interest in this paper. Denote by $R^\circ$ the set of ring elements not contained in any minimal prime of $R$. When $R$ is reduced, we will denote by $R^{1/p^e}$ the ring of $p^e$\ts{th} roots of elements in $R$ inside an algebraic closure of the total ring of fractions of $R$. Note that the $e$\ts{th} iterate $F^e$ of the Frobenius map may be identified with the natural inclusion $R \hookrightarrow R^{1/p^e}$, which in turn can be identified with the map $R\to F^e_* R$ to the restriction of scalars along $F^e$.
\begin{definition}
    Let $(R,\fm)$ be a $d$-dimensional reduced $F$-finite local ring. We say that $R$ is
\begin{enumerate}[label=(\alph*)]
    \item \emph{$F$-injective} if, for each $i\leq d$, the induced Frobenius map $F\colon H^i_\fm(R)\to H^i_\fm(R)$ on local cohomology is injective;
    \item \emph{$F$-rational} if $R$ is Cohen--Macaulay and the tight closure of zero
    \[
    0^*_{H^d_\fm(R)} = \{\eta \in H^d_\fm(R) \mid \exists c \in R^\circ \text{ such that } cF^e(\eta) = 0\text{ for all }e\gg0\}
    \]
    is the zero submodule of $H^d_\fm(R)$;
    \item \emph{$F$-pure} if the inclusion $F\colon R\hookrightarrow R^{1/p}$ splits, i.e. if there exists an $R$-linear map $\varphi\colon R^{1/p}\to R$ such that $\varphi\circ F = \id_R$;
    \item \emph{strongly $F$-regular} if for each $c\in R^\circ$, there exists an integer $e\gg 0$ such that the map $R\rightarrow R^{1/p^e}$ sending $1\to c^{1/p^e}$ splits.
\end{enumerate}
\end{definition}

We will have occasional use for equivalent formulations of some of the notions above in terms of ideal closures. Let $(R,\fm)$ be an $F$-finite reduced local ring as above and let $I\subseteq R$ be an ideal. For each prime power $q=p^e$ denote $I^{[q]}:=F^e(I)R$. Then the \emph{Frobenius closure of $I$}, denoted $I^F$, is the ideal consisting of ring elements $r\in R$ for which $r^q\in I^{[q]}$ for all $q\gg 0$. Similarly, the \emph{tight closure} of $I$, denoted $I^*$, is the ideal
$$I^* = \{r\in R\mid \exists c\in R^\circ\text{ such that } cr^{q}\in I^{[q]}\text{ for all } q\gg 0\}.$$ Then $R$ is $F$-pure if and only if $I=I^F$ for all ideals $I\subseteq R$ \cite{Hoc77}. Similarly, $R$ is $F$-rational if and only if $R$ is Cohen--Macaulay and $I = I^*$ for some (or all) ideal(s) $I\subseteq R$ generated by a system of parameters \cite{HH94a}.

In the $F$-finite case all the notions introduced above are known to be open conditions which localize well (see e.g. \cite[Section 6]{MP25} for a summary), and the definitions thus extend in the obvious way to arbitrary noetherian rings and schemes.

We next recall the following regular (resp. smooth) base change results for $F$-purity and $F$-rationality.
\begin{theorem}\label{thm:F-sing-base-change}
    Let $(R,\fm)\to (S,\fn)$ be a flat local ring homomorphism between excellent local rings of prime characteristic $p>0$. Then:
\begin{enumerate}[label=(\roman*)]
    \item If $R$ is $F$-pure and $S/\fm S$ is a regular local ring (or more generally, Gorenstein and $F$-pure), then $S$ is $F$-pure;\label{thm:F-pure-base-change}
    \item If $R$ is $F$-rational and $S/\fm S$ is geometrically regular over $R/\fm$, then $S$ is $F$-rational. \label{thm:F-rat-base-change}
\end{enumerate}
\end{theorem}
\begin{proof}
    There are many references for this; see e.g. \cite[Proposition 5.4]{Ma14} or \cite[Theorem 7.4]{MP25} for \ref{thm:F-pure-base-change} and \cite[Theorem 3.1]{Vel95} or \cite[Theorem 7.8]{MP25} for \ref{thm:F-rat-base-change}. See also \cite{Ene00}.
\end{proof}

\subsection{The dual of Frobenius and \texorpdfstring{$F$}{F}-rationality}\label{Subsection dual of Frobenius}

Our analysis of $F$-rationality in a family relies upon a well-known dual characterization of $F$-rationality presented in terms of Cartier linear maps of a canonical module for which we lack a convenient reference outside of the local case. If $R$ is an $F$-finite domain of prime characteristic $p>0$, then $R$ is the homomorphic image of an $F$-finite regular ring $S$ and $R\cong S/\fP$ for some prime ideal $\fP\in \Spec(S)$ by \cite{Gab04}, and the duals of the Frobenius maps $F^e:R\to R$ are defined relative to the onto map $S\to R$.  If $h= \dim(S_{\fP})$ then $\Ext_S^h(R,S)$ and $\Ext^h_S(F^e_*R,S)$ are identifications of global canonical modules of $R$ and $F^e_*R$ respectively. The rings $R$ and $F^e_*R$ are abstractly isomorphic and we may identify $\Ext_S^h(F^e_*R,S) \cong F^e_*\omega_R$; see \cite[Remark 6.5]{MP25}. The resulting map $\Tr^e\colon F^e_*\omega_R\to \omega_R$ (sometimes denoted $\Tr^e_R$, or $\Tr^e_{R/S}$ if we wish to specify the ring $R$ or the rings $R$ and $S$), is obtained by applying $\Ext^h_S(-,S)$ to $R\to F^e_*R$, and is referred to as the \emph{$e$\ts{th} Frobenius dual map}. The $e$\ts{th} Frobenius dual map is defined upon choice of presentation $S\to R$ with $S$ regular, but is locally unique up to isomorphism. Indeed, if $\fp\in \Spec(R)$, then by local duality, the Matlis dual of $(\Tr^{e}_R)_{\fp}\colon F^{e}_*\omega_{R_{\fp}}\to \omega_{R_{\fp}}$ is the Frobenius map of local cohomology modules $H^{\Height(\fp)}_{\fp R_{\fp}}(R_{\fp})\xrightarrow{F^e}H^{\Height(\fp)}_{\fp R_{\fp}}(F^e_*R_{\fp})\cong F^e_*H^{\Height(\fp)}_{\fp R_{\fp}}(R_{\fp})$. Therefore the completion of $(\Tr^{e}_R)_{\fp}$ is the Matlis dual of the $e$\ts{th} Frobenius map on the top local cohomology of $R_\fp$, hence $(\Tr^{e}_R)_{\fp}$ is uniquely determined up to isomorphism.

If $R$ is an $F$-finite domain, then the $F$-rationality of $R$ can be determined through examining pre-multiplication of an iterate of the Frobenius dual by a (parameter) test element. If $c$ is a parameter test element of a $d$-dimensional local ring $(R,\fm,k)$, then $R$ is $F$-rational if and only if $R$ is Cohen--Macaulay and the $0$-submodule of $H^d_{\fm}(R)$ is tightly closed by \cite[Proposition~2.5 and Theorem~2.6]{Smi97}. By standard arguments, $0^*_{H^d_{\fm}(R)} = 0$ if and only if there exists a parameter test element $c\in R^{\circ}$ and $e\in\NN$ so that $$H^d_{\fm}(R)\xrightarrow{F^e}F^e_*H^d_{\fm}(R)\xrightarrow{\cdot F^e_*c}F^e_*H^d_{\fm}(R)$$ is injective. See \cite[Proof of Proposition~6.3]{MP25} for more details. Parameter test elements are known to exist in our setting by \cite[Theorem~5.10]{HH94a}.

A detailed treatment of the following in the local case can be found in \cite[Proposition~6.3]{MP25}. We provide the necessary details for the non-local case.

\begin{proposition}[{cf. \cite{Vel95}}]
     \label{prop: global Frobenius trace}
    Let $R$ be an $F$-finite Cohen--Macaulay domain of prime characteristic $p>0$ and $\Tr^e\colon F^e_*\omega_R\to \omega_R$ an identification of the $e$\ts{th} Frobenius dual defined in \cref{Subsection dual of Frobenius}. Let $c\in R^{\circ}$ be a parameter test element of $R$. Then $R$ is $F$-rational if and only if there exists $e_0\in\NN$ so that 
    \[
    \Tr^{e_0}_R(F^{e_0}_*c-) : F^{e_0}_*\omega_R\xrightarrow{\cdot F^{e_0}_*c}F^{e_0}_*\omega_R\xrightarrow{\Tr^{e_0}}\omega_R
    \]
    is surjective.
\end{proposition}

\begin{proof}
    Let $\fp\in\Spec(R)$. The local ring $R_{\fp}$ is $F$-rational if and only if $R_{\fp}$ is Cohen--Macaulay and $$0^*_{H^{\Height(\fp)}_{\fp R_{\fp}}(R_{\fp})} = 0$$ by \cite[Proposition~2.5 and Theorem~2.6]{Smi97}. If $c\in R^{\circ}$ is a parameter test element then $R_{\fp}$ is $F$-rational if and only if there exists $e_{\fp}\in \NN$ so that $$H^{\Height(\fp)}_{\fp R_{\fp}}(R_{\fp})\xrightarrow{F^{e_{\fp}}} F^{e_{\fp}}_*H^{\Height(\fp)}_{\fp R_{\fp}}(R_{\fp})\xrightarrow{\cdot F^{e_{\fp}}_*c} F^{e_{\fp}}_*H^{\Height(\fp)}_{\fp R_{\fp}}(R_{\fp})$$ is injective. By Matlis duality, $R_{\fp}$ is $F$-rational if and only if there exists $e_{\fp}\in \NN$ so that 
    \[
    \Tr^{e_\fp}(F^{e_{\fp}}_*c-)_{\fp}\colon F^{e_{\fp}}_*\omega_{R_{\fp}}\xrightarrow{\cdot F^{e_{\fp}}_*c}F^{e_{\fp}}_*\omega_{R_{\fp}} \xrightarrow{\Tr^{e_{\fp}}}\omega_{R_{\fp}}
    \]
    is surjective. Hence if there exists $e_0\in\NN$ so that $\Tr^{e_0}_R(F^{e_{0}}_*c-)$ is surjective globally, then $R$ is $F$-rational. 
    
    Conversely, if $R$ is $F$-rational, then for each $\fp\in\Spec(R)$ there exists $e_{\fp}\in\NN$ so that $\Tr^{e_{\fp}}(F^{e_\fp}_*c-)_{\fp}$ is surjective. Equivalently, the cokernel of $\Tr^{e_{\fp}}(F^{e_{\fp}}_*c-)$ is not supported at $\fp$. Therefore there exists an open neighborhood $U$ of $\fp\in\Spec(R)$ so that for all $\fq\in U$ the localized map $\Tr^{e_{\fp}}(F^{e_{\fp}}_*c-)_{\fq}$ is surjective. If $e\geq e_{\fp}$ then $\Tr^e(F^e_*c-)_{\fq}$ is surjective. By quasi-compactness of $\Spec(R)$, there exists an $e_0\in \NN$ so that $\Tr^{e_0}(F^{e_0}_*c-)$ is surjective locally, and therefore globally.
\end{proof}

\section{A modification of the Cumino--Greco--Manaresi axioms}\label{sec:CGM}

The article \cite{CGM86} provides an axiomatic framework for establishing Bertini theorems for a local property $\cP$ of a quasi-projective variety over an algebraically closed field. The main result of \cite{CGM86} is that the following axioms for a local property $\cP$ of locally noetherian schemes imply Bertini theorems for $\cP$:
\begin{enumerate}[label=(A\arabic*)]
    \item Let $\varphi\colon Y \to Z$ be a flat morphism with regular fibers. If $Z$ has local property $\cP$, then so does $Y$. \label{CGM-1}
    \item Let $\rho\colon Y \to S$ be a finite type morphism of schemes where $Y$ is excellent and $S$ is integral with generic point $\eta \in S$. If $Y_\eta$ is geometrically $\cP$, then there exists an open neighborhood $\eta \in U \subseteq S$ such that $Y_s$ is geometrically $\cP$ for every $s \in U$. \label{CGM-2}
    \item The local property $\cP$ is open on schemes of finite type over a field. \label{CGM-3}
\end{enumerate}

Axiom \ref{CGM-2} has been studied extensively for various local properties $\cP$; see, for example, \cite[Appendix~E.1]{GW20} and \cite[\S9]{EGA}. 

A first approach to \cref{thm:Bertini} would be to verify that the local property $\cP =$ ``$F$-rational and $F$-pure'' satisfies \ref{CGM-1}, \ref{CGM-2}, and \ref{CGM-3}. Indeed, \ref{CGM-3} holds since for a quasi-projective variety $X$ over an algebraically closed field, both the $F$-rational and $F$-pure loci are open subsets of $X$. The axiom \ref{CGM-2} is satisfied by $F$-purity by \cite[Corollary~5.2]{SZ13}.

However, it is not clear whether the local property $\cP=$ ``$F$-rational and $F$-pure'' satisfies \ref{CGM-1}. Without the $F$-pure hypothesis, $F$-rationality can fail to satisfy \ref{CGM-1} by \cite{QGSS24}. We remark, however, that axiom \ref{CGM-1} is stronger than what is generally necessary to prove Bertini theorems; see \cite[Remark (i)]{CGM86}. Based on this observation, we propose the following variation of the axioms \ref{CGM-1} and \ref{CGM-2}:

\begin{enumerate}[label=(B\arabic*)]
\setcounter{enumi}{-1}
    \item If $Y$ is a scheme of finite type over a field $K$, and $Y$ is $\cP$, then so is $Y \times_K \Spec(K[x])$. \label{CGM-B0}
    \item Let $Y$ be a scheme of finite type over a field $K$, and let $K \subseteq L$ be a finite field extension. Suppose $X$ is a noetherian excellent scheme such that there exists a flat morphism $Y \times_K L \to X$ with regular fibers. If $X$ and $Y$ are $\cP$ then so is $Y \times_K L$. \label{CGM-B1}
    
    \item Let $Y \to S$ be a finite type morphism of $F$-finite schemes where $S$ is integral with generic point $\eta \in S$. Suppose additionally that $S$ is of finite type over an algebraically closed field. If $Y_\eta$ is geometrically $\cP$, then there exists an open neighborhood $\eta \in U \subseteq S$ such that $Y_s$ has local property $\cP$ for every closed point $s \in U$. \label{CGM-B2}
\end{enumerate}

We now show that the axioms \ref{CGM-B0}, \ref{CGM-B1}, \ref{CGM-B2}, together with \ref{CGM-3}, yield Bertini theorems as in \cite[Theorem 1, Corollary 1, Corollary 2]{CGM86}.

\begin{theorem}\label{thm:CGM}
    Let $\cP$ be a local property of schemes, and let $k=\overline{k}$ be an algebraically closed field of characteristic $p>0$. For a scheme $Z$, let $\cP(Z)$ denote the locus of points of $Z$ satisfying $\cP$.
    \begin{enumerate}[label=(\alph*)]
    \item \label{thm:CGM-1} Let $X$ be a scheme of finite type over $k$. Let $\phi\colon X\to\PP^n_k$ be a morphism with separably generated (not necessarily algebraic) residue field extensions. Suppose $\cO_{X,x}$ satisfies $\cP$ for all $x\in X$ and that $\cP$ satisfies \ref{CGM-B0}, \ref{CGM-B1} and \ref{CGM-B2}. Then there exists a nonempty open subset $U\subseteq (\PP^n_k)^\vee$ such that the local rings of $\phi^{-1}(H)$ satisfy $\cP$ for all $H\in U$.
        \item \label{thm:CGM-2} Let $V$ be an algebraic variety over $k$ and let $S$ be a finite-dimensional linear system on $V$. Assume that the map $V\dashrightarrow\PP^n_k$ corresponding to $S$ induces separably generated field extensions wherever it is defined. If $\cP$ satisfies \ref{CGM-B0}, \ref{CGM-B1}, \ref{CGM-B2}, and \ref{CGM-3}, then the general element of $S$, considered as a subscheme of $V$, satisfies $\cP$ except perhaps at the base points of $S$ and at the complement of $\cP(V)$.\label{cgm-thm-2}
        \item  \label{thm:CGM-3} Let $X\subseteq \PP^n_k$ be a closed embedding where $X$ satisfies $\cP$ and $\cP$ satisfies \ref{CGM-B0}, \ref{CGM-B1} and \ref{CGM-B2}. Then for a general hyperplane $H\in (\PP^n_k)^\vee$, $X\cap H$ satisfies $\cP$. If $\cP$ additionally satisfies \ref{CGM-3}, then we have the containment of $\cP$-loci $\cP(X\cap H)\supseteq \cP(X)\cap H$ for a general hyperplane $H\in (\PP^n_k)^\vee$.
    \end{enumerate}
\end{theorem}

\begin{proof}
   We start by proving \ref{thm:CGM-1}. Let $\PP = \PP^n_k$ and $\PP^{\vee} = (\PP^n_k)^{\vee}$. We abuse notation as follows: if $H\subseteq \PP$ is a hyperplane, we still denote by $H\in \PP^{\vee}$ the corresponding closed point of the dual space. Choose coordinates $X_0,X_1,\ldots,X_n$ and $Y_0,Y_1,\ldots, Y_n$ of the coordinate rings of $\PP$ and $\PP^{\vee}$ respectively. Let $Z$ be the Zariski closure of $\{(x,H)\in \PP\times_k \PP^{\vee}\mid x\in H\}$, defined by the equation $X_0Y_0+\cdots + X_nY_n = 0$. Consider the commutative diagram:
    \[
    \begin{xymatrix}
    {
    X\times_\PP Z \ar[rr]^{\sigma}\ar[dd]\ar[dr]^{\rho} & & Z\ar[d]\ar[dl]^{\pi}\\
    \, & \PP^{\vee} & \PP\times_k \PP^{\vee}\ar[l]\ar[d]\\
    X\ar[rr]^{\phi} & \, & \PP
    }
    \end{xymatrix}
    \]
and note that $\phi^{-1}(H)\cong \sigma^{-1}(\pi^{-1}(H))$. We first show that $X \times_\PP Z_\eta$ is geometrically $\cP$ over $\kappa(\eta)$, that is, that $(X \times_\PP Z_\eta) \times_{\kappa(\eta)}L$ is $\cP$ for all finite field extensions $\kappa(\eta) \subseteq L$.

    The collection of affine charts $\{(D(X_a) \times D(Y_b)) \cap Z \mid a\not=b\}$ covers the projective variety $Z$. The affine charts $\{D(X_a)\}$ cover $\PP$, and therefore $\{\phi^{-1}(D(X_a))\}$ is an open cover of $X$. By symmetry, we may assume $a=0$ and $b=n$. Let $x_i = \frac{X_i}{X_0}$ and $y_j= \frac{Y_j}{Y_n}$, and note that 
    \[
    Z_{0,n}:=(D(X_0) \times D(Y_n)) \cap Z  = \Spec\left( \frac{k[x_1,x_2,\ldots,x_n,y_0,y_1,y_2,\ldots,y_{n-1}]}{(y_0 + \sum_{i=1}^{n-1}x_iy_i + x_{n})}\right).
    \]
    The preimage $\phi^{-1}(D(X_0))$ in $X$ is covered by affine open sets of finite type over $k$, say $\{V_1,V_2,\ldots,V_m\}$. It suffices to show each $V_i\times_{\Spec(k[x_1,\ldots,x_n])} (Z_{0,n})_{\eta}$ is geometrically $\cP$  over $\kappa(\eta)$. So let $V=\Spec(R)$ be an affine open subset of $\phi^{-1}(D(X_0))$, with $R$ an algebra of finite type over $k$. Let $\phi^*:k[x_1,\ldots,x_n] \to R$ be the induced map of $k$-algebras, and for $j=1,\ldots,n$ let $f_j=\phi^*(x_j)$. Note that
    \[
    V \times_{\PP} Z_{0,n} = \Spec\left(\frac{R[y_0,\ldots,y_{n-1}]}{(y_0+\sum_{i=1}^{n-1}f_iy_i + f_n)}\right) = \Spec(R[y_1,\ldots,y_{n-1}]),
    \]
    which is $\cP$ thanks to \ref{CGM-B0}. Since $\cP$ also localizes, we conclude that $V \times_{\PP} (Z_{0,n})_\eta$ is $\cP$ as well. Note that, if we set $W=k[y_0,\ldots,y_{n-1}]\smallsetminus \{0\}$, then
    \[
    V \times_\PP (Z_{0,n})_\eta = \Spec\left(\left(\frac{R[y_0,\ldots,y_{n-1}]}{(y_0+\sum_{i=1}^{n-1} f_iy_i + f_n)}\right)_W\right)
    \]
    is a scheme of finite type over $\kappa(\eta) = k[y_0,\ldots,y_{n-1}]_W$. Since the residue field extensions of $\phi$ are separably generated by assumption, by \cite{CGM86} (see Lemma 2 and the proof of Theorem 1) we have that the fibers of the map 
    \[
    R \to \left(\frac{R[y_0,\ldots,y_{n-1}]}{(y_0+\sum_{i=1}^{n-1} f_iy_i + f_n)}\right)_W \otimes_{\kappa(\eta)} L
    \]
    are regular for all finite extensions $\kappa(\eta) \subseteq L$. By \ref{CGM-B1} we have that $\left(\frac{R[y_0,\ldots,y_{n-1}]}{(y_0+\sum_{i=1}^{n-1} f_iy_i + f_n)}\right)_W \otimes_{\kappa(\eta)} L$ is $\cP$, and thus we conclude that $X \times_\PP Z_\eta$ is geometrically $\cP$ over $\kappa(\eta)$. 
    Now, since $\cP$ satisfies \ref{CGM-B2} there exists an open subset $U\subseteq\PP^\vee$ such that $\rho^{-1}(s)$ is $\cP$ for every closed point $s \in U$. Each closed point of $U$ corresponds to a hyperplane, and we thus conclude that $\phi^{-1}(H)$ is $\cP$ for a general hyperplane $H \subseteq \PP$. This concludes the proof of \ref{thm:CGM-1}, from which the proofs of \ref{thm:CGM-2} and \ref{thm:CGM-3} follow analogously to the corresponding ones in \cite{CGM86}.
\end{proof}

\section{\texorpdfstring{$F$}{F}-rational and \texorpdfstring{$F$}{F}-pure singularities} \label{sec:FratFpure}

We consider the local property $\cP =$ ``$F$-rational and $F$-pure'' for noetherian $F$-finite schemes. Observe that $\cP$ satisfies axioms \ref{CGM-B0} (e.g., by \cref{thm:F-sing-base-change}) and \ref{CGM-3} \cite[Theorem 3.2]{Has10}, \cite[Theorem 1.1]{Vel95} (see also \cite[Section 6]{MP25}).

\subsection{\texorpdfstring{\ref{CGM-B1}}{(B1)} for \texorpdfstring{$\cP$ = ``$F$-rational and $F$-pure"}{P = "F-rational and F-pure"}}
The following lemma and its corollary show that the local property ``$F$-rational and $F$-pure'' satisfies a slight strengthening of \ref{CGM-B1}.

\begin{lemma} \label{B1:F-rational-F-pure} Let $K \subseteq L$ be a finite extension of fields of characteristic $p>0$, and let $R$ be a $K$-algebra of finite type. Consider the map $R \to S:=R \otimes_K L$, and assume that $R$ is $F$-rational and that $S$ is $F$-pure. Then $S$ is $F$-rational.
\end{lemma}

\begin{proof}
Note that the extension $R \subseteq S$ is finite and free. Let $\fn$ be a maximal ideal of $S$, so that $\fm:=\fn \cap R$ is a maximal ideal of $R$. Note that $\fn$ is a minimal prime of $\fm S$, and set $(A,\fm):=R_\fm \to S_\fn=:(B,\fn)$. Let $d$ be the common Krull dimension of the local rings $A$ and $B$. Note that, since $A$ is $F$-rational, local, and essentially of finite type over a field, it is a Cohen--Macaulay domain.
    
    In order to show $F$-rationality of $B$, first assume that the extension $K \subseteq L$ is finite separable. Then $R\to R\otimes_K L$ is \'{e}tale, hence so is the localization $A\to B$ and the result follows from \cref{thm:F-sing-base-change}\ref{thm:F-rat-base-change}. 

    Now assume instead that the finite field extension $K \subseteq L$ is purely inseparable. In particular, there exists $e_{0}\in\NN$ such that $L^{p^{e_0}} \subseteq K$. Then $S^{p^{e_0}} \subseteq R$, and 
    we also get that $B^{p^{e_0}} \subseteq A$. Combined with the fact that $B$ is reduced since it is $F$-pure, this shows that $B$ is a domain. Moreover, since $A\to B$ is flat with relative dimension zero, $B$ is Cohen--Macaulay.
    
    Let $\fq = (x_1,\dots, x_d)\subseteq A$ be a common system of parameters for $A$ and $B$. The $F$-rationality of $B$ will follow from the equality $\fq B =(\fq B)^*$. To that end, let $y\in (\fq B)^*$. There exists $c\in B^{\circ}$ so that
    \begin{align*}
        cy^{p^e} \in (x_1^{p^e},\dots, x_d^{p^e})B
    \end{align*}
    for all $e\gg 0$. Since $c^{p^{e_0}}$ and $y^{p^{e_0}}\in A$, we have by faithful flatness of $A\to B$,
    \begin{align*}
        c^{p^{e_0}} y^{p^{e+e_0}}\in (x_1^{p^{e+e_0}},\dots, x_d^{p^{e+e_0}})B \cap A = (x_1^{p^{e_0}},\ldots,x_d^{p^{e_0}})^{[p^e]}A
    \end{align*}
    for all $e\gg 0$. It follows that
    \begin{align*}
        y^{p^{e_0}}\in \left((x_1^{p^{e_0}},\dots, x_d^{p^{e_0}})A\right)^* = (x_1^{p^{e_0}},\dots, x_d^{p^{e_0}})A\subseteq (x_1^{p^{e_0}},\dots, x_d^{p^{e_0}})B
    \end{align*}
    by flatness and the assumption that $A$ is $F$-rational. Since $B$ is $F$-pure, it follows that $y\in \fq B$, and thus $B$ is $F$-rational as claimed. 
    
    To handle the general case, write the finite extension $K \subseteq L$ as $K \subseteq L_1 \subseteq L$, where $K \subseteq L_1$ is separable and $L_1 \subseteq L$ is purely inseparable. Since $S = R \otimes_K L \cong (R \otimes_K L_1) \otimes_{L_1} L$ we conclude that $S$ is $F$-rational by combining the two cases discussed above.
\end{proof}

\begin{corollary}
    \label{B1:F-rational-F-pure-v2}
    Let $Y$ be a scheme of finite type over an $F$-finite field $K$ of prime characteristic $p>0$ and let $K\subseteq L$ be a finite field extension. Let $X$ be a noetherian $F$-finite scheme and suppose that $g:Y\times_KL\to X$ is flat with regular fibers. If $X$ is $F$-pure and $Y$ is $F$-rational, then $Y\times_K L$ is $F$-rational and $F$-pure.
\end{corollary}

\begin{proof}
    For each point $g(y)=x$, apply \cref{thm:F-sing-base-change}\ref{thm:F-pure-base-change} to $\mathcal{O}_{X,x}\to \mathcal{O}_{Y\times_K L,y}$. Since $F$-purity can be checked locally, this shows that $Y\times_K L$ is $F$-pure. We can now apply \cref{B1:F-rational-F-pure} locally to $Y\times_K L\to Y$ to conclude that $Y\times_K L$ is also $F$-rational.
\end{proof}

\subsection{\texorpdfstring{\ref{CGM-B2}}{(B2)} for \texorpdfstring{$\cP$ = ``$F$-rational"}{P = "F-rational"}}

The goal of this subsection is to prove that $F$-rationality satisfies \ref{CGM-B2} up to an integrality assumption on the generic fiber; see \cref{cor:A2-F-rational}.

In this subsection we will repeatedly be in the following setting. While not every result will assume this setting, we will always be able to reduce to it via inverting an element of the source ring $A$.
\begin{setting}\label{setting-B2}
    Let $\varphi\colon A \to R$ be a finite type flat homomorphism of integral domains, where $A$ is regular and $R$ is Cohen--Macaulay. We let $K=\Frac(A)$ be the fraction field of $A$ and, for $\fp \in \Spec(A)$, we let $\kappa(\fp) = (A/\fp)_\fp$. If $B$ is an $A$-algebra, we denote $R_B:=R\otimes_A B$. If $x\in\Spec(R_B)$, we denote by $R_{B,x}$ the localization of $R_B$ at $x$. Let $T=A[X_1,\ldots,X_n]$ be a polynomial algebra mapping onto $R$, and set $\fP := \ker(T \twoheadrightarrow R)$. Further assume that $A$ is finitely generated over a perfect field $k$ and $R$ is geometrically $F$-rational over the generic point of $\Spec(A)$.
\end{setting}

We recall the fact that $F$-purity satisfies axiom \ref{CGM-2}, hence axiom \ref{CGM-B2} as well.

\begin{theorem}[{\cite[Corollary 5.2]{SZ13}}]\label{thm:a-2-F-pure}
Let $Y\to S$ be a finite type morphism of $F$-finite schemes where $S$ is integral with generic point $\eta\in S$. If $Y_\eta$ is geometrically $F$-pure, then there exists an open neighborhood $\eta\in U\subseteq S$ such that $Y_s$ is (geometrically) $F$-pure for every $s\in U$.
\end{theorem}

\begin{remark}
    In what follows, we will refer to Frobenius pushforwards via both the restriction of scalars functor and via rings of $p$\ts{th}-power roots depending in part on if we wish to emphasize the module structure. This occasionally leads to expressions involving both operations $(-)^{1/p^e}$ and $F^e_*(-)$ simultaneously.
\end{remark}

As in the approaches taken in \cite{PSZ18} to study $F$-singularities in families, we utilize the theory of relative canonical modules to establish $R_{L}$ is $F$-rational for all finite extensions $\kappa(\fm)\subseteq L$ and $\fm$ belonging to a non-empty set of maximal ideals $U\subseteq \maxSpec(A)$. By \cref{prop: global Frobenius trace}, it suffices to identify a non-empty set of maximal ideals $U\subseteq \maxSpec(A)$, an element $c\in R$ that serves as a parameter test element of $R_{L}$ for all finite extensions $L$ of $\kappa(\fm)$ and $\fm\in U$, and a constant $e_0$ so that the Frobenius dual $\Tr^{e_0}_{R_{L}}(F^{e_0}_*c-)\colon F^{e_0}_*\omega_{R_{L}}\to \omega_{R_{L}}$ is surjective.

If $R\cong A[x_1,x_2,\ldots,x_n]/\fP$, $h=\Height(\fP)$, then $\omega_{R|A} := \Ext^h_T(R,T)$ is \emph{a relative canonical module for $R$ over $A$}. For all $e \geq 1$ note that $T_{A^{1/p^e}}=A^{1/p^e}[X_1,\ldots,X_n]$ and that, in this case, $\omega_{R^{1/p^e}|A^{1/p^e}} = \Ext^h_{T_{A^{1/p^e}}}(R^{1/p^e},T_{A^{1/p^e}})$. The first lemma concerns the base change of relative canonical modules to fibers (see \cite[\href{https://stacks.math.columbia.edu/tag/0E9M}{Tag 0E9M}]{stacks-project}).

\begin{lemma} \label{lemma base change} Let $\varphi:A \to R$ be as in \cref{setting-B2}. For all $\fp \in \Spec(A)$ we have
\[
\omega_{R^{1/p^e}|A^{1/p^e}} \otimes_{F^e_*A} \kappa(\fp)^{1/p^e} \cong \Ext^h_{T_{\kappa(\fp)^{1/p^e}}}(F^e_*(R_{\kappa(\fp)}),T_{\kappa(\fp)^{1/p^e}}) = \omega_{F^e_*(R_{\kappa(\fp)})|\kappa(\fp)^{1/p^e}},
\]
where $T_{\kappa(\fp)^{1/p^e}} \cong \kappa(\fp)^{1/p^e}[X_1,\ldots,X_n]$. Moreover, $\omega_{F^e_*(R_{\kappa(\fp)})|\kappa(\fp)^{1/p^e}}$ is a canonical module for $F^e_*(R_{\kappa(\fp)})$.
\end{lemma}
\begin{proof}
    Note that $F^e_*A\to F^e_*R$ is still flat and of finite type. The lemma then follows from \cite[Tags \href{https://stacks.math.columbia.edu/tag/0BZV}{0BZV} and \href{https://stacks.math.columbia.edu/tag/0BZW}{0BZW}]{stacks-project}.
\end{proof}

Continuing as in \cref{setting-B2}, for all $e > 0$ let
\[
\Tr^e_{R|A}\colon \Ext^h_{T_{A^{1/p^e}}}(R^{1/p^e},T_{A^{1/p^e}}) \to \Ext^h_{T_{A^{1/p^e}}}(R_{A^{1/p^e}},T_{A^{1/p^e}})
\]
be the $\Ext^h_{T_{A^{1/p^e}}}(-,T_{A^{1/p^e}})$-dual of the relative Frobenius map $R_{A^{1/p^e}} \to R^{1/p^e}$. By \cref{lemma base change}, for any $\fp \in \Spec(A)$ such that $\kappa(\fp)$ is perfect we have that $\Tr^e_{R|A} \otimes_{F^e_*A} \kappa(\fp)^{1/p^e}$ coincides with
\[
\Tr^e_{R_{\kappa(\fp)}|\kappa(\fp)}\colon \Ext^h_{T_{\kappa(\fp)}}(F^e_*(R_{\kappa(\fp)}),T_{\kappa(\fp)}) \to \Ext^h_{T_{\kappa(\fp)}}(R_{\kappa(\fp)},T_{\kappa(\fp)}),
\]
which identifies with the $\Ext^h_{T_{\kappa(\fp)}}(-,T_{\kappa(\fp)})$-dual of the Frobenius map $R_{\kappa(\fp)} \to F^e_*(R_{\kappa(\fp)})$. In particular at every maximal ideal $\fm$ of $T_{\kappa(\fp)}$ such a map coincides with the Matlis dual of the $e$\ts{th} iterate of the Frobenius map on local cohomology, $F^e\colon H^d_{\fm}(R_{\kappa(\fp)}) \to H^d_{\fm}(R_{\kappa(\fp)})$. If $c \in R$ is a parameter test element for $R_{\kappa(\fp)}$, and $\kappa(\fp)$ is perfect, we then see that $R_{\kappa(\fp)}$ is $F$-rational if and only if $\Tr^e_{R|A}(F^e_*c -) \otimes_{A^{1/p^e}} \kappa(\fp)^{1/p^e}$ is surjective for some $e >0$. Note that, if $cF^e\colon H^d_{\fm}(R_{\kappa(\fp)}) \to H^d_{\fm}(R_{\kappa(\fp)})$ is injective for some $e$, then in particular $F$ itself is injective, and thus $F^{e'}(cF^e) = c^{p^{e'}}F^{e+e'}$ is injective for all $e' \geq 0$. But then $cF^{e+e'}$ is injective for all $e' \geq 0$ as well.

The following lemma is inspired by \cite[Proposition 2.6]{HH00}.

\begin{lemma} \label{key lemma}
Let $A\to R$ be as in \cref{setting-B2} and write $R_K = K[X_1,\ldots, X_n]/\fQ$ for a prime $\fQ$ of height $h:=\Ht \fQ$. Assume that $R_K$ is geometrically $F$-rational over $K$. Then for any $0\neq c\in R_K$, there exists $e_0>0$ such that the map 
    \[
    \Tr^e_{R_K|K}(F^e_*c-)\colon \Ext^h_{T_{K^{1/p^e}}}(F^e_*(R_K),T_{K^{1/p^e}}) \to \Ext^h_{T_{K^{1/p^e}}}(R_{K^{1/p^e}},T_{K^{1/p^e}})
    \]
    is surjective for all $e \geq e_0$.
\end{lemma}

\begin{proof}
Let $\fm$ be a maximal ideal of $R_K$. Note that, since $R_K$ is a domain of finite type over a field, every maximal ideal has height $d=n-h$. In particular, by prime avoidance we can find an ideal $J \subseteq \fm$, generated by $d$ elements, which is contained only in (finitely many) maximal ideals of $R_K$. In particular, $J$ is a parameter ideal of $R_K$, the ring $R_K/JR_K$ is Artinian, and a finite dimensional $K$-vector space. Now note that $JR_{K^{1/p^e}}$ is a parameter ideal of $R_{K^{1/p^e}}$ for all $e >0$, and $R_{K^{1/p^e}}/JR_{K^{1/p^e}}$ is Artinian and a finite-dimensional $K^{1/p^e}$-vector space as well. Now let $e >0$ be fixed, and for all $e'>0$ we let $$I_{e,e'}: = \{x \in R_{K^{1/p^e}} \mid cx^{p^{e'}} \in (JR_{K^{1/p^e}})^{[p^{e'}]}\}.$$ Since each ideal $(J R_{K^{1/p^e}})^{[p^{e'}]}$ is Frobenius closed, one can readily check that
\begin{align}
J R_{K^{1/p^e}} \subseteq \ldots \subseteq I_{e,e'+1} \subseteq I_{e,e'} \subseteq \ldots \subseteq I_{e,1} \subseteq R_{K^{1/p^e}},\label{lem:key lemma 1}
\end{align}
and since $R_{K^{1/p^e}}/JR_{K^{1/p^e}}$ is Artinian, there exists $e_0(e,J)$ such that $I_{e,e'} = I_{e,e_0}$ for all $e' \geq e_0$.

It is clear that the stable value is $(JR_{K^{1/p^e}})^* = JR_{K^{1/p^e}}$. This in particular means that, if $x \in R_{K^{1/p^e}}$ is such that $cx^{p^{e'}} \in (JR_{K^{1/p^e}})^{[p^{e'}]}$ for a single $e' \geq e_0$, then $x \in JR_{K^{1/p^e}}$. In terms of local cohomology, this gives that $cF^{e'}\colon  H^{d}_J(R_{K^{1/p^e}}) \to H^{d}_J(R_{K^{1/p^e}})$ is injective for all $e' \geq e_0$, i.e., localizing at any maximal ideal $\fm$ containing $J$, the map $cF^{e'}\colon H^d_\fm(R_{K^{1/p^e}}) \to H^d_\fm(R_{K^{1/p^e}})$ is injective for all $e' \geq e_0$. Note that, since $R_K \to R_{K^{1/p^e}}$ is purely inseparable, there is a one-to-one correspondence between maximal ideals $\fm$ of $R_K$ and $\fm_e$ of $R_{K^{1/p^e}}$ for all $e$. Moreover, $\sqrt{\fm R_{K^{1/p^e}}} = \fm_e$, therefore $H^i_{\fm}(R_{K^{1/p^e}}) \cong H^i_{\fm_e}(R_{K^{1/p^e}})$ for all $e$.

We now show that $e_0$ can be chosen independently of $e$ (but still depending on $J$, for the moment). Letting $\kperf$ denote the perfect closure of $K$, the same argument above as in \eqref{lem:key lemma 1} shows that the descending chain of ideals $$I_{\infty,e'}=\{x\in R_{\kperf}\mid cx^{p^{e'}}\in (JR_{\kperf})^{[p^{e'}]}\}$$ stabilizes at $JR_{\kperf}$. By faithful flatness, we may then choose $$e_0(e,J)\leq \min\{e'\geq 1\mid I_{\infty,e'}=JR_{\kperf}\}<\infty$$ which is independent of $e$ as claimed. To summarize, for any maximal ideal $\fm$ of $R_K$ there exists $e_0=e_0(\fm)$ such that $$cF^{e}\colon  H^d_\fm(R_{K^{1/p^e}}) \to H^d_\fm(R_{K^{1/p^e}})$$ is injective for all $e \geq e_0$. Dually, the map 
$$\Phi_{e}\colon  \Ext^h_{T_{K^{1/p^e}}}(F^{e}_*(R_{K^{1/p^e}}),T_{K^{1/p^e}}) \to \Ext^h_{T_{K^{1/p^e}}}(R_{K^{1/p^e}},T_{K^{1/p^e}})$$ is surjective for all $e \geq e_0(\fm)$ when localized at $\fm_e$. 
For all $e>0$ let $M_{e} = \coker(\Phi_e)$, and let $U_{e} = \max\Spec(R_{K^{1/p^e}}) \smallsetminus \Supp(M_e)$. By definition, $\Phi_e$ is surjective when localized at any maximal ideal of $U_e$. Moreover, since $M_e$ is a finitely generated $R_{K^{1/p^e}}$-module we have that $\Supp(M_e)$ is a closed subset of $\Spec(R_{K^{1/p^e}})$, and thus $U_e$ is an open subset of $\max\Spec(R_{K^{1/p^e}})$. Using the homeomorphism
\begin{align}
\maxSpec(R_{K^{1/p^e}}) \cong \max\Spec(R_K),  \fm_e \longleftrightarrow \fm\label{lem: key lemma 2}
\end{align} we can view each set $U_e$ as an open subset of $\max\Spec(R_K)$. 
\begin{claim} For all $e$ we have $U_e \subseteq U_{e+1}$. 
\end{claim}
\begin{proof}[Proof of claim]
    Let $\fm \in U_e$, so that $cF^e\colon H^d_{\fm}(R_{K^{1/p^e}}) \to H^d_{\fm}(R_{K^{1/p^e}})$ is injective. We want to show that $$cF^{e+1}\colon H^d_\fm(R_{K^{1/p^{e+1}}}) \to H^d_\fm(R_{K^{1/p^{e+1}}})$$ is injective as well, so that $\fm \in U_{e+1}$. To see this, let $J$ be a parameter ideal of $S:=R_{K,\fm}$. To simplify notation in what follows, for each $e'$ denote $$S_{e'}:=R_{K^{1/p^{e'}},\fm_{e'}}$$ again using the identification \eqref{lem: key lemma 2}. Recall that $JS_e$ is a parameter ideal of $S_e$, and by assumption we have that if $x \in S_e$ is such that $cx^{p^e} \in (JS_e)^{[p^e]}$, then $x \in JS_e$. Indeed, this is equivalent to $cF^e\colon H^d_\fm(R_{K^{1/p^e}}) \to H^d_\fm(R_{K^{1/p^e}})$ being injective. Now assume that $x=\sum(x_i \otimes \lambda_i^{1/p^{e+1}}) \in R_{K^{1/p^{e+1}}}$ is such that its image in $S_{e+1}$ satisfies $cx^{p^{e+1}} = c\sum (x_i^{p^{e+1}} \otimes \lambda_i) \in (JS_{e+1})^{[p^{e+1}]}$. Let $\pi\colon S_{e+1} \to S_e$ be the map induced by the splitting $K^{1/p^{e+1}} \to K^{1/p^e}$ of the natural inclusion $K^{1/p^e} \subseteq K^{1/p^{e+1}}$, and note that 
\[
\pi(cx^{p^{e+1}}) = \pi(\sum (cx_i^{p^{e+1}} \otimes \lambda_i)) = \sum (cx_i^{p^{e+1}} \otimes \lambda_i) =  c \left(\sum (x_i^p \otimes \lambda_i^{1/p^e})\right)^{p^e}
\]
belongs to the ideal
\[
\pi((JS_{e+1})^{[p^{e+1}]}) = \pi(J^{[p^{e+1}]}S_{e+1})= J^{[p^{e+1}]}\pi(S_{e+1}) \subseteq (JS_e)^{[p^{e+1}]}.
\]
As pointed out above, the condition $c \left(\sum (x_i^p \otimes \lambda_i^{1/p^e})\right)^{p^e} \in \left((JS_e)^{[p]}\right)^{[p^e]}$ gives that $\sum (x_i^p \otimes \lambda_i^{1/p^e}) \in (JS_e)^{[p]}$. If we map this through the natural inclusion $\iota\colon  S_e \to S_{e+1}$ induced by $K^{1/p^e} \subseteq K^{1/p^{e+1}}$, we obtain that $$\iota(\sum (x_i^p \otimes \lambda_i^{1/p^e})) = \sum (x_i^p \otimes \lambda_i^{p/p^{e+1}}) = (\sum(x_i\otimes \lambda_i^{1/p^{e+1}}))^p =x^p \in \iota(J^{[p]}S_e) \subseteq (JS_{e+1})^{[p]}$$ so that $x \in JS_{e+1}$ since $JS_{e+1}$ is Frobenius closed as $S_{e+1}$ is $F$-rational. This shows that $$cF^{e+1}\colon H^d_{\fm}(R_{K^{1/p^{e+1}}}) \to H^d_{\fm}(R_{K^{1/p^{e+1}}})$$ is injective, proving the claim.
\end{proof}

The claim shows that $U_1 \subseteq \ldots \subseteq U_e \subseteq U_{e+1} \subseteq \ldots$ in  $\max\Spec(R_K)$ is an ascending chain of open subsets of the noetherian topological space $\max\Spec(R_K)$, which must then stabilize. Given that, for any $\fm \in \maxSpec(R_K)$, there exists $e_0$ such that $\fm \in U_e$ for all $e \geq e_0$, we conclude that there exists $e_0>0$ such that $U_{e_0} = \maxSpec(R_K)$. In other words, the map
$$\Phi_e\colon \Ext^h_{T_{K^{1/p^e}}}(F^e_*(R_{K^{1/p^e}}),T_{K^{1/p^e}}) \to \Ext^h_{T_{K^{1/p^e}}}(R_{K^{1/p^e}},T_{K^{1/p^e}})$$ is surjective for all $e \geq e_0$. Recall that $R_{K^{1/p^e}} = R\otimes_A K^{1/p^e}$. To conclude, note that we can factor the Frobenius $R_{K^{1/p^e}}\to F^e_*(R_{K^{1/p^e}})$ as
\begin{equation*}
    \begin{tikzcd}
        R_{K^{1/p^e}}\arrow[r,"F^e"]\arrow[d] & F^e_*(R_{K^{1/p^e}})\\
        R^{1/p^e}\otimes_{A^{1/p^e}} K^{1/p^e}\arrow[r,"\cong"] & F^e_*(R_K)\arrow[u]
    \end{tikzcd}
\end{equation*}
showing that the induced map $\Phi_e$ factors through $$\Tr^e_{R_K|K}(F^e_*c -)\colon  \Ext^h_{T_{K^{1/p^e}}}(F^e_*(R_K),T_{K^{1/p^e}}) \to \Ext^h_{T_{K^{1/p^e}}}(R_{K^{1/p^e}},T_{K^{1/p^e}}).$$ This map is then surjective for all $e \geq e_0$ as desired.
\end{proof}

\begin{theorem}
    \label{theorem:A2}
   Let $k$ be a perfect field of prime characteristic $p>0$ and $A$ an integral domain finitely generated over $k$. Suppose $\varphi\colon A\to R$ is a finite type morphism of $F$-finite domains. If $R_K$ is geometrically $F$-rational over $K$, where $K$ is the fraction field of $A$, then there exists a non-empty open subset $U\subseteq \max\Spec(A)$ such that for all $\fm\in U$ one has $R_{\kappa(\fm)}$ is (geometrically) $F$-rational over $\kappa(\fm)$.
\end{theorem}
\begin{proof}
    If $\fm$ is a maximal ideal of $A$, then $k\to A/\fm$ is a separable extension. By \cref{thm:F-sing-base-change}\ref{thm:F-rat-base-change}, it suffices to show there is an open subset $U\subseteq \max\Spec(A)$ so that, for all $\fm\in U$, $R_{\kappa(\fm)}$ is $F$-rational.

    The regular locus of $A$ and the flat locus of $\varphi$ are open, so we can replace $A$ by a principal localization and assume $A$ is regular and $\varphi$ is flat. Assume that $T = A[X_1,\ldots, X_n]$ is a polynomial extension of $A$ mapping onto $R$ and $R\cong T/\fP$. Let $h$ be the height of $\fP$. By generic freeness, \cite[\href{https://stacks.math.columbia.edu/tag/051S}{Tag 051S}]{stacks-project}, there exists $0\not = a\in A$ so that $\Ext_T^{h+i}(R, T)_a = 0$ for all $i\ne 0$. We can replace $A$ by $A_{a}$ and assume $R$ is Cohen--Macaulay. Since $R_K$ is geometrically normal, by \cite[Corollaire 9.9.5]{EGA} we can further shrink $A$ to assume that all fibers of $\varphi$ are geometrically normal and so that $R$ is also a normal domain by \cite[Theorem 23.9]{Mat86}. 

    We next require an element $c\in R$ that serves as a test element of each $R_{L}$ ranging over all finite extensions $L\supseteq \kappa(\fm)$ and all $\fm$ belonging to a non-empty open set of maximal ideals $U\subseteq \maxSpec(A)$. The existence of such an element is outlined in \cite[Section~5.2]{PSZ18} and reproduced here for convenience. The Jacobian ideal $\Jac_A(R)$ of $R$ over $A$ is nonzero as $R_K$ is geometrically reduced, and commutes with base change by \cite[Discussion~4.4.7]{SH06} and \cite[Proposition~16.4]{Eis95}. If $0\not=c\in\Jac_A(R)$, then by generic freeness we can replace $A$ by a principal localization and assume the image of $c$ is a nonzerodivisor in $R_{\kappa(\fm)}$ for all maximal ideals $\fm$ of $A$. Even further, $k\subseteq \kappa(\fm)$ is a finite field extension with $k$ perfect, therefore $\kappa(\fm)$ is also perfect. It follows that for all maximal ideals $\fm$ of $A$ and all finite extensions $\kappa(\fm)\subseteq L$, the image of $c$ is a nonzero element of the Jacobian ideal of $R_{L}$ over the field $L$ and therefore is a test element of $R_{L}$ by \cite[Theorem on Page 217]{Hoc07}.

    For each $e\geq 1$ note that $T_{A^{1/p^e}}=A^{1/p^e}[X_1,\ldots,X_n]$. Let
    \[
    \Tr^e_{R|A}\colon  \Ext^h_{T_{A^{1/p^e}}}(R^{1/p^e},T_{A^{1/p^e}}) \to \Ext^h_{T_{A^{1/p^e}}}(R_{A^{1/p^e}},T_{A^{1/p^e}})
    \]
    be the map defined above. Let $c\in R$ be an element that serves as a test element of $R_{\kappa(\fm)}$ for all maximal ideals $\fm$ of $\Spec(A)$ as described above. By assumption, $R_K$ is geometrically $F$-rational over $K$, therefore by \cref{key lemma} there exists $e>0$ such that 
    \[
    \Tr^e_{R_K|K}(F^e_*c -) \colon  \Ext^h_{T_{K^{1/p^e}}}(F^e_*(R_K),T_{K^{1/p^e}}) \to \Ext^h_{T_{K^{1/p^e}}}(R_{K^{1/p^e}},T_{K^{1/p^e}})
    \]
    is surjective. Note that, since $F^e_*(R_K) \cong R^{1/p^e} \otimes_{A^{1/p^e}} K^{1/p^e}$ and $R^{1/p^e}$ is a finite $T_{A^{1/p^e}}$-module, we have that 
    \begin{align*}
    \Ext^h_{T_{K^{1/p^e}}}(F^e_*(R_K),T_{K^{1/p^e}})  &\cong \Ext^h_{T_{A^{1/p^e}} \otimes_{A^{1/p^e}} K^{1/p^e}}(F^e_*R \otimes_{A^{1/p^e}} K^{1/p^e},T_{A^{1/p^e}}\otimes_{A^{1/p^e}} K^{1/p^e})\\
    &\cong \Ext^h_{T_{A^{1/p^e}}}(R^{1/p^e},T_{A^{1/p^e}})_W,
    \end{align*}
 where $W$ is the multiplicatively closed set $A \smallsetminus \{0\}$; here we use that $(F^e_*A)_W \cong F^e_*(A_W) = K^{1/p^e}$. Note then that $\Tr^e_{R_K|K}(F^e_*c -) = \left(\Tr^e_{R|A}(F^e_*c -)\right)_W$ is surjective. By generic freeness, there exists $0 \ne a\in A$ so that the cokernel of $\Tr^e_{R|A}(F^e_*c -)$ localizes to zero. Thus after inverting a non-zero element $a\in A$ we may assume that $\Tr^e_{R|A}(F^e_*c -)_a$ is surjective. Let $U = \max\Spec(A_a)$ which is then identified as a non-empty Zariski-open subset of $\max\Spec(A)$. If $\fm \in U$,  by right exactness of tensor products, we have that 
    \[
    \Tr^e_{R|A}(F^e_*c-)_a \otimes_{F^e_*A} F^e_*(A/\fm) \cong \Tr^e_{R|A}(F^e_*c-) \otimes_{F^e_*A} F^e_*(A/\fm)
    \]
    is still surjective. Note that the isomorphism follows from the fact that $$F^e_*(A/\fm)_a \cong F^e_*((A/\fm)_a) \cong F^e_*(A/\fm)$$ since $a$ is already a unit in $A/\fm = \kappa(\fm)$. By Lemma \ref{lemma base change}, using also the fact that $\kappa(\fm)$ is perfect, we conclude that
    \[
    \Tr^e_{R_{\kappa(\fm)}|\kappa(\fm)}(F^e_*c -)\colon  \Ext^h_{T_{\kappa(\fm)}}(F^e_*(R_{\kappa(\fm)}),T_{\kappa(\fm)}) \to \Ext^h_{T_{\kappa(\fm)}}(R_{\kappa(\fm)},T_{\kappa(\fm)})
    \]
    is surjective. As already observed, at every maximal ideal $\fn$ of $R_{\kappa(\fm)}$ such a map is the Matlis dual of $c F^e\colon  H^{\Height(\fn)}_{\fn}(R_{\kappa(\fm)}) \to H^{\Height(\fn)}_{\fn}(R_{\kappa(\fm)})$, which is then injective. 
    We conclude that $R_{\kappa(\fm)}$ is (geometrically) $F$-rational for all $\fm \in U$.
\end{proof}

As a corollary, we obtain \ref{CGM-B2} up to an integrality assumption on $Y_\eta$; compare with \cite[Theorem~5.13]{PSZ18}.

\begin{corollary}\label{cor:A2-F-rational}
    Let $Y\to S$ be a finite type morphism of $F$-finite schemes where $S$ is integral with generic point $\eta\in S$. Suppose additionally that $S$ is finite type over an algebraically closed field. If $Y_\eta$ is integral and geometrically $F$-rational over $\kappa(\eta)$, then there exists an open neighborhood $\eta\in U\subseteq S$ such that $Y_s$ is $F$-rational for every closed point $s\in U$.
\end{corollary}
\begin{proof}
    The statement is local, so we assume that $Y = \Spec(R)$ and $S=\Spec(A)$, where $A$ is a domain with fraction field $K=\Frac(A)$ and $R$ is of finite type over $A$. After possibly inverting a non-zero element of $A$, we may assume that $R$ is a subring of $R_K = R \otimes_A K$. By assumption, $R_K$ is a domain and hence so is $R$. The statement now follows from \cref{theorem:A2}.
\end{proof}

\section{Bertini theorems for \texorpdfstring{$F$}{F}-rational and \texorpdfstring{$F$}{F}-pure singularities} \label{sec:Bertini}

The following results constitute \cref{thm: Bertini for F-rational singularities}; this will follow almost immediately from \cref{thm:CGM} together with the results of the previous section.

\begin{theorem}[Bertini Theorems for $F$-purity + $F$-rationality]\label{thm:Bertini}
    Let $X$ be a variety over an algebraically closed field $k=\overline{k}$ of characteristic $p>0$. 
    \begin{enumerate}[label=(\alph*)]
        \item If $\phi\colon X\to \mathbb{P}^n_k$ is a $k$-morphism with separably generated (not necessarily algebraic) residue field extensions and $X$ is $F$-rational and $F$-pure, then there exists a non-empty open subset $U\subseteq (\mathbb{P}_k^n)^{\vee}$ so that for each hyperplane $H\in U$, $\phi^{-1}(H)$ is $F$-rational and $F$-pure.\label{thm:Bertini-1}
        \item If $X\subseteq \mathbb{P}^n_k$ is a closed immersion and $X$ is $F$-rational and $F$-pure, then $X\cap H$ is $F$-rational and $F$-pure for a general hyperplane $H\subseteq \PP^n_k$.\label{thm:Bertini-2}
        \item More generally, if $X\subseteq \mathbb{P}^n_k$ and $U_X\subseteq X$ is the $F$-rational $F$-pure locus of $X$, then for a general hyperplane $H\in (\mathbb{P}_k^n)^{\vee}$, $U_X\cap H\subseteq U_{X\cap H}.$\label{thm:Bertini-3}
    \end{enumerate}
\end{theorem}

\begin{proof}
    Let $\phi\colon X\to \PP^n_k=:\PP$ be as in \ref{thm:Bertini-1}, and let $Z$ denote the reduced closed subscheme of $\PP \times_k \PP^\vee$ given by the Zariski-closure of the incidence correspondence
\begin{equation*}
    \{(x,H)\in \PP \times_k \PP^\vee\mid x\in H\}.
\end{equation*}
In the proof of \cref{thm:CGM}, the axiom \ref{CGM-B0} is applied to $Y:=X \times_\PP Z$, while \ref{CGM-B1} is applied to the map $Y_{\eta}\times_{\kappa(\eta)} L\to X$ for (finite) extensions $\kappa(\eta)\subseteq L$. Moreover, \ref{CGM-B2} is applied to $\rho\colon Y\to \PP^\vee=: S$. Thus we may assume that all schemes in axioms \ref{CGM-B0}-\ref{CGM-B2} and \ref{CGM-3} are $F$-finite. By the above, in \ref{CGM-B2} we may assume that $S$ is finite type over an algebraically closed field.

We have already observed at the beginning of \cref{sec:FratFpure} that the local property $\cP=$ ``$F$-rational and $F$-pure'' satisfies \ref{CGM-B0} and \ref{CGM-3}. By \cref{B1:F-rational-F-pure-v2}, it also satisfies \ref{CGM-B1}. Now note that $X$ is normal and irreducible; by \cite[Proposition 1.5.10 and Corollary 3.4.14]{FOV99}, the same holds for $\rho^{-1}(\eta) = Y_\eta$. Thus, it suffices to prove \ref{CGM-B2} where $Y_\eta$ is further assumed to be integral, and this follows from \cref{cor:A2-F-rational} and \cref{thm:a-2-F-pure}. This concludes the proof of part \ref{thm:Bertini-1}, thanks to \cref{thm:CGM}\ref{thm:CGM-1}. 

The proof of the other statements follows just as in \cite{CGM86} from \cref{thm:CGM}.
\end{proof}

\begin{corollary}[cf. {\cite[Corollary 1]{CGM86}}]{\label{cor:bertini-second}}
    Let $V$ be an algebraic variety over $k=\overline{k}$ and let $S$ be a finite dimensional linear system on $V$. Assume that the map $V\dashrightarrow\PP^n_k$ corresponding to $S$ induces (whenever defined) separably generated field extensions. Then the general element $H$ of $S$, viewed as a subscheme of $V$, is $F$-rational and $F$-pure except perhaps at the base points of $S$ and at the points of $V$ which are not $F$-rational or $F$-pure.
\end{corollary}

\section{On \texorpdfstring{$\Q$}{Q}-Gorenstein \texorpdfstring{$F$}{F}-Pure Non-\texorpdfstring{$F$}{F}-Regular Rings}
The purpose of this section is to partially clarify the types of varieties to which \cref{thm: Bertini for F-rational singularities} might be applied. Recall that by \cite{SZ13}, \cref{thm: Bertini for F-rational singularities} only has new content when $X$ is \emph{not} strongly $F$-regular. When the ambient normal Cohen--Macaulay variety $X\subseteq\PP^n$ with canonical divisor $K_X$ is $\Q$-Gorenstein (that is, when $nK_X$ is Cartier for some $n>0$), \cref{prop: char cong to 1 mod p} imposes restrictions on its index.

We briefly review the necessary language of divisors employed in the statement and the proof of the next theorem. Let $(R,\fm)$ be a local normal domain with fraction field $K=\Frac(R)$, and let $D$ be a Weil divisor on $\Spec(R)$. We denote the corresponding divisorial ideal by $$R(D):=\{x\in K\smallsetminus\{0\}\mid \Div(x)+D\geq 0\}\cup\{0\},$$ that is, the global sections of $\cO_{\Spec R}(D)$. A canonical divisor of $X=\Spec(R)$ is a Weil divisor $K_X$ so that $R(K_X)\cong \omega_R$ is a canonical module of $R$. By a $\Q$-Gorenstein ring, we mean the following:
\begin{definition}
    A local normal domain $(R,\fm)$ with a canonical divisor $K_X$ of $X=\Spec(R)$ is \emph{$\Q$-Gorenstein} if there exists an integer $n>0$ so that $nK_X \sim 0$, i.e., $R(nK_X)\cong R$. If $R$ is $\Q$-Gorenstein, then the \emph{$\Q$-Gorenstein index of $R$} is the smallest positive integer $n>0$ so that $nK_X\sim 0$.
\end{definition}

\begin{theorem}[{\cite[Corollary 2.2]{Pol22}}] Let $(R,\fm)$ be a local normal domain and let $M$ be a finitely generated ($S_2$) $R$-module. If $D_1,\ldots, D_t$ are divisors representing distinct elements of the divisor class group of $R$ and such that $R(D_i)$ is a direct summand of $M$ for each $i$, then $\bigoplus\limits_{i=1}^t R(D_i)$ is also a direct summand of $M$.\label{thm:direct sums}
\end{theorem}

\begin{theorem}
    \label{prop: char cong to 1 mod p}
    Let $(R,\fm,k)$ be an $F$-finite normal domain of prime characteristic $p>0$ and $\mathbb{Q}$-Gorenstein of index $n$. If $R$ is $F$-pure, not $F$-regular, and if $R_{\fp}$ is $F$-regular for all non-maximal prime ideals $\fp\in\Spec(R)$ (for example, if $R$ has an isolated singularity), then $p\equiv 1\pmod{n}$.
\end{theorem}

\begin{proof}
    For each $e\in\NN$ let
    \[
    I_e(\fm) = \{r\in R\mid R\xrightarrow{\cdot F^e_*r}F^e_*R\mbox{ does not split}\} 
    \]
    denote the $e$\ts{th} Frobenius splitting ideal of $R$. Because $R$ is $F$-pure and not strongly $F$-regular, the splitting prime of $R$ is $\fP = \cap_{e\in \NN}I_e(\fm)$ and is a proper nonzero prime ideal of $R$ so that for all $\fp\in V(\fP)$, $R_{\fp}$ is not strongly $F$-regular (see \cite[Theorem~3.3, Corollary~3.4, and Proposition~3.6]{AE05}). Because $R_{\fp}$ is assumed to be strongly $F$-regular for all non-maximal primes $\fp\in\Spec(R)$, $\fP = \fm$ and for each $e\in\NN$, $I_e(\fm) = \fm$. In particular, $F^e_*R$ has exactly $\dim_k(F^e_*R/F^e_*\fm) = [k^{1/p^{e}}:k]$ free $R$-summands for all $e\in\NN$. 
    
    Assume that $n = p^{e_0}m$ and $p$ does not divide $m$. If $e$ is sufficiently large and divisible, then $m$ divides $p^e-1$ by Fermat's Little Theorem, $n$ divides $p^{e_0}(p^e-1)$, and $p^{e_0}(p^e-1)K_R\sim 0$. We first show that $e_0=0$, that is $(p^e-1)K_R\sim 0$. Consider a direct sum decomposition $F^{e_0}_*R\cong R^{\oplus [k^{1/p^{e_0}}:k]}\oplus -$. If we tensor with $R(-(p^e-1)K_R)$ and reflexify, we obtain a direct sum decomposition
    \[
    F^{e_0}_*R(-p^{e_0}(p^e-1)K_R)\cong F^{e_0}_*R\cong R(-(p^e-1)K_R)^{\oplus [k^{1/p^{e_0}}:k]}\oplus -.
    \]
    If $R(-(p^e-1)K_R)\not\cong R$, i.e., $(p^e-1)K_R\not\sim 0$, then by \cref{thm:direct sums} there exists a direct sum decomposition
    \[
    F^{e_0}_*R\cong R^{\oplus [k^{1/p^{e_0}}:k]}\oplus R(-(p^e-1)K_R)^{\oplus [k^{1/p^{e_0}}:k]}\oplus -.
    \]
    If we apply $\Hom_R(-,R)$ to a direct sum decomposition $F^e_*R\cong R^{\oplus [k^{1/p^{e}}:k]}\oplus -$, then the standard chain of isomorphisms
    \begin{align*}
        \Hom_R(F^e_*R,R) &\cong \Hom_R(F^e_*R(p^eK_R),R(K_R))\\
        & \cong F^e_*\Hom_R(R(p^eK_R), R(K_R))\\
        & \cong F^e_*R(-(p^e-1)K_R)
    \end{align*}
    yields a direct sum decomposition
    \[
    F^e_*R(-(p^e-1)K_R) \cong  \Hom_R(R^{\oplus [k^{1/p^{e}}:k]}\oplus -, R) \cong R^{\oplus [k^{1/p^{e}}:k]}\oplus -.
    \]
    Consequently, $F^{e+e_0}_*R$ has a direct sum decomposition
    \begin{align*}
        F^{e+e_0}_*R & \cong F^e_*F^{e_0}_*R \\
        & \cong F^e_*\left(R^{\oplus [k^{1/p^{e_0}}:k]}\oplus R(-(p^e-1)K_R)^{\oplus [k^{1/p^{e_0}}:k]}\oplus -\right) \\
        & \cong F^e_*R^{\oplus [k^{1/p^{e_0}}:k]} \oplus F^e_*R(-(p^e-1)K_R)^{\oplus [k^{1/p^{e_0}}:k]}\oplus -\\
        & \cong R^{\oplus [k^{1/p^{e}}:k][k^{1/p^{e_0}}:k]}\oplus R^{\oplus [k^{1/p^{e}}:k][k^{1/p^{e_0}}:k]}\oplus -\\
        & = R^{\oplus 2[k^{1/p^{e+e_0}}:k]}\oplus -.
    \end{align*}
    That is, $F^{e+e_0}_*R$ has at least $2[k^{1/p^{e+e_0}}:k]$ free summands, contradicting that $F^{e+e_0}_*R$ only has $[k^{1/p^{e+e_0}}:k]$ free summands for all $e+e_0\geq 1$. Therefore if $e$ is sufficiently large and divisible, then $-(p^e-1)K_R\sim 0$, i.e., $K_R$ has index relatively prime to $p$.

    To show that $(p-1)K_R\sim 0$ we choose $e\geq 2$ so that $(p^e-1)K_R\sim 0$, i.e., $p^eK_R\sim K_R$. We then consider a direct sum decomposition of $F^{e-1}_*R\cong R^{\oplus [k^{1/p^{e-1}}:k]}\oplus -$, tensor with $R(pK_R)$, and then reflexify to find a direct sum decomposition
    \[
    F^{e-1}_*R(p^eK_R)\cong F^{e-1}_*R(K_R)\cong R(pK_R)^{\oplus [k^{1/p^{e-1}}:k]}\oplus -.
    \]
    We then apply $\Hom_R(-,R(K_R))$ to the above direct sum decomposition to find that
    \[
    \Hom_R(F^{e-1}_*R(K_R),R(K_R))\cong F^{e-1}_*R\cong R(-(p-1)K_R)^{\oplus [k^{1/p^{e-1}}:k]}\oplus -.
    \]
    If $R(-(p-1)K_R)\not\cong R$, i.e., $(p-1)K_R\not\sim 0$, then by \cref{thm:direct sums} once more there is a direct sum decomposition
    \[
    F^{e-1}_*R\cong R^{\oplus [k^{1/p^{e-1}}:k]}\oplus R(-(p-1)K_R)^{\oplus [k^{1/p^{e-1}}:k]}\oplus -.
    \]
    As above, $F_*R(-(p-1)K_R)\cong R^{\oplus [k^{1/p}:k]}\oplus -$ has $[k^{1/p}:k]$ free $R$-summands, hence
    \begin{align*}
        F^e_*R & \cong F_*F^{e-1}_*R \\
        & \cong F_*\left(R^{\oplus [k^{1/p^{e-1}}:k]}\oplus R(-(p-1)K_R)^{\oplus [k^{1/p^{e-1}}:k]}\oplus -\right)\\
        & \cong F_*R^{\oplus [k^{1/p^{e-1}}:k]}\oplus F_*R(-(p-1)K_R)^{\oplus [k^{1/p^{e-1}}:k]}\oplus -\\
        & \cong R^{\oplus [k^{1/p^{e-1}}:k][k^{1/p}:k]}\oplus R^{\oplus [k^{1/p^{e-1}}:k][k^{1/p}:k]}\oplus -\\
        & = R^{\oplus 2[k^{1/p^{e}}:k]} \oplus -.
    \end{align*}
    Therefore $F^e_*R$ admits at least $2 [k^{1/p^{e}}:k]$ free $R$-summands, a contradiction as $F^e_*R$ only admits $[k^{1/p^{e}}:k]$ free summands for all $e\in\NN$. Therefore if $n$ is the torsion index of $K_R$, then $(p-1)K_R\sim 0$ implies $n|p-1$. Equivalently, $p\equiv 1\pmod{n}$.
\end{proof}
A curious corollary is that under the hypotheses of \cref{prop: char cong to 1 mod p}, $F$-rationality forces the characteristic to be odd.
\begin{corollary}\label{cor:f-rational-not-f-regular}
    Suppose that $R$ satisfies the hypotheses of \cref{prop: char cong to 1 mod p} and additionally is $F$-rational. Then $p>2$. In particular, if $(R,\fm)$ is a two-dimensional $F$-finite local ring of characteristic two which is $F$-pure and $F$-rational, then $R$ is $F$-regular.
\end{corollary}
\begin{proof}
    The first statement follows at once from \cref{prop: char cong to 1 mod p} and from the equivalence between $F$-regularity and $F$-rationality in the Gorenstein setting \cite[Theorem 4.2(g)]{HH94a}. To see the second statement, note that $R$ is normal by \cite[Theorem 4.2(b)]{HH94a} and therefore has an isolated singularity. An $F$-rational ring $R$ has pseudo-rational singularities by \cite{Smi97}. $F$-finite rings are excellent by \cite{Kun76}, and excellent surfaces admit resolutions of singularities by \cite{Lip78}; therefore $R$ is a two-dimensional rational singularity, and the divisor class group of $R$ is finite by results of \cite{Lip69}. In particular, the canonical divisor class is torsion (i.e., $R$ is $\Q$-Gorenstein) and the result follows from \cref{prop: char cong to 1 mod p} once more.
\end{proof}

\begin{remark}\cref{prop: char cong to 1 mod p} should be compared with Watanabe's classification of $F$-singularities of surfaces \cite{Wat91}. To summarize, let $R = \bigoplus_{n\geq 0}R_n$ be a two-dimensional $\N$-graded normal domain with degree zero part $R_0 = k$ an algebraically closed field of characteristic $p>0$ which is $F$-pure, $F$-rational, but not $F$-regular. It is proven in \cite[Theorem 4.2(2ii)]{Wat91} that $R$ arises as the section ring $$R\cong \bigoplus_{n\geq 0}H^0(\PP^1_k,\cO_{\PP^1_k}(nD)) T^n$$ where $D$ is a rational coefficient Weil divisor on $\PP^1_k$ with an associated list of allowable characteristics. Since they are $F$-rational, the section rings are $\Q$-Gorenstein (see the proof of \cref{cor:f-rational-not-f-regular}) and moreover have isolated singularities by Serre's criterion for normality. It may be checked that the restrictions on the characteristic $p$ exactly match those of \cref{prop: char cong to 1 mod p} with respect to the index.
\end{remark}

\printbibliography

\end{document}